\documentclass[11pt]{amsart}
\usepackage[all]{xy}
\usepackage{mathrsfs}
\usepackage[margin=1.3in]{geometry}

\usepackage{graphicx}
\usepackage[cyr]{aeguill}\usepackage[francais,english]{babel}
\usepackage{amsmath,amsfonts,amssymb}
\usepackage{color}
\usepackage[utf8]{inputenc}
\usepackage{comment}
\usepackage[shortlabels]{enumitem}
\usepackage{etoolbox}
\usepackage{float}
\usepackage{latexsym}
\usepackage{lipsum}
\usepackage{needspace}
\usepackage{tikz}
\usepackage{hyperref}

\newtheorem{theorem}{Theorem}
\numberwithin{theorem}{section}
\numberwithin{equation}{section}
\newtheorem{lemma}[theorem]{Lemma}

\newtheorem{proposition}[theorem]{Proposition}
\newtheorem{definition}[theorem]{Definition}

\newtheorem{claim}[theorem]{Claim}

\begin{document}

\title{On a Cheeger type inequality in Cayley graphs of finite groups }
\author{Arindam Biswas}

\address{Universit\"at Wien, Fakult\"at f\"ur Mathematik\\
	Oskar-Morgenstern-Platz 1, 1090 Wien, Austria. }
\email{biswasa43@univie.ac.at}

\keywords{Cheeger inequality, expander graphs, finite Cayley graphs}

\begin{abstract}
  Let $G$ be a finite group. It was remarked in \cite{myfav54} that if the Cayley graph $C(G,S)$ is an expander graph and is non-bipartite then the spectrum of the adjacency operator $T$ is bounded away from $-1$. In this article we are interested in explicit bounds for the spectrum of these graphs. Specifically, we show that the non-trivial spectrum of the adjacency operator lies in the interval $\left[-1+\frac{h(\mathbb{G})^{4}}{\gamma}, 1-\frac{h(\mathbb{G})^{2}}{2d^{2}}\right]$, where $h(\mathbb{G})$ denotes the (vertex) Cheeger constant of the $d$ regular graph $C(G,S)$ with respect to a symmetric set $S$ of generators and $\gamma = 2^{9}d^{6}(d+1)^{2}$.
	
\end{abstract}

\maketitle	
	
\section{Introduction}
Throughout this article we will consider a finite group $G$ with $|G|=n$. We will denote by $C(G,S)$ for a symmetric subset $S\subset G$ of size $|S|=d$, to be the Cayley graph of $G$ with respect to $S$. Then $C(G,S)$ is $d$ regular. Given a finite $d$ regular Cayley graph $C(G,S)$, we have the normalised adjacency matrix $T$ of size $n\times n$ whose eigenvalues lie in the interval $[-1,1]$. The normalised Laplacian matrix of $C(G,S)$ denoted by $L$ is defined as 
\begin{equation}\label{2.1}
L := I_{n} - T,
\end{equation}
where $I_{n}$ denotes the identity matrix. The eigenvalues of $L$ lie in the interval $[0,2]$. \\
 It is easy to see that $1$ is always an eigenvalue of $T$ and $0$ that of $L$. We denote the eigenvalues of $T$ as $-1\leqslant t_{n}\leqslant ... \leqslant t_{2}\leqslant t_{1} = 1$ and that of $L$ as $\lambda_{i}=1-t_{i}, i=1,2,...,n$. The graph $C(G,S)$ is connected if and only if $\lambda_{2} >0$ (equivalently $t_{2}<1$). The graph is bipartite if and only if $\lambda_{n} = 2$. \\

We recall the notion of Cheeger constant. 
\begin{definition}[Vertex boundary of a set]
	Let $\mathbb{G} = (V,E)$ be a graph with vertex set $V$ and edge set $E$. For a subset $V_{1}\subset V$, let $N(V_{1})$ denoting the neighbourhood of $V_{1}$ be
	$$N(V_{1}) := \lbrace v\in V : vv_{1}\in E \text{ for some } v_{1}\in V_{1}\rbrace.$$
	Then the boundary of $V_{1}$ is defined as $ \delta(V_{1}) := N(V_{1})\backslash V_{1}$.
\end{definition}

\begin{definition}[Cheeger constant]\label{vertex-Cheeger}
	The Cheeger constant of the graph $\mathbb{G} = (V,E)$, denoted by $h(\mathbb{G})$ is defined as 
	$$h(\mathbb{G}) := \inf \lbrace \frac{|\delta(V_{1})|}{|V_{1}|} : V_{1}\subset V, |V_{1}|\leqslant \frac{|V|}{2}\rbrace.$$
	This is also called the vertex Cheeger constant of a graph.
\end{definition}

\begin{definition}[$(n,d,\epsilon)$ expander]\label{vexp}
	Let $\epsilon>0$. An $(n,d,\epsilon)$ expander is a graph $(V,E)$ on $|V| =n$ vertices, having maximal degree $d$, such that for every set $V_{1}\subseteq V$ satisfying $|V_{1}|\leqslant \frac{n}{2}$, $|\delta(V_{1})|\geqslant \epsilon|V_{1}|$ holds (equivalently, $h(\mathbb{G})\geqslant \epsilon).$
\end{definition}

In this article, we are interested in the spectrum of the Laplace operator $L$ for the Cayley graph $C(G,S)$. The Cayley graph is bipartite if and only if there exists an index two subgroup $H$ of $G$ which is disjoint from $S$. See Proposition \ref{prop}. It was observed in \cite{myfav54}(Appendix E) that if $C(G,S)$ is an expander graph and is non-bipartite, then the spectrum of $T$ is not only bounded away from $1$ but also from $-1$. Here we show that 

\begin{theorem}\label{mainthm}
	Let the Cayley graph $C(G,S)$ be an expander with $|S| = d$ and $h(G)$ denote its Cheeger constant. Then if $C(G,S)$ is non-bipartite, we have
	$$\lambda_{n}\leqslant 2 - \frac{h(G)^{4}}{\alpha d^{6}(d+1)^{2}}, $$
	where $\lambda_{n}$ is the largest eigenvalue of the normalised Laplacian matrix and $\alpha$ is an absolute constant (we can take $\alpha = 2^{9}$).
\end{theorem}

The strategy of the proof closely follows the combinatorial arguments of Breuillard--Green--Guralnick--Tao in \cite{myfav54}.

\section{Proofs}
There are two notions of expansion in graphs - the vertex expansion as in Definition \ref{vexp} and the edge expansion.
\begin{definition}[Edge expansion]
		Let $\mathbb{G} = (V,E)$ be a $d$-regular graph with vertex set $V$ and edge set $E$. For a subset $V_{1}\subset V$, let $E(V_{1},V\backslash V_{1})$ be the edge boundary of $V_{1}$, defined as
	$$E(V_{1},V\backslash V_{1}) := \lbrace (v_{1},s)\in E: v_{1}\in V, v_{1}s\in V\backslash V_{1} \rbrace .$$
	Then the edge expansion ratio $\phi(V_{1})$ is defined as 
	$$\phi(V_{1}) := \frac{|E(V_{1},V\backslash V_{1})|}{d|V_{1}|}.$$
\end{definition}
\begin{definition}[Edge-Cheeger constant]\label{Edge-Cheeger}
	The edge-Cheeger constant denoted by $\mathfrak{h}(\mathbb{G})$ is 
	$$\mathfrak{h}(\mathbb{G}):= \inf_{V_{1}\subset V, |V_{1}|\leqslant |V|/2} \phi(V_{1}).$$
\end{definition}

In a $d$ regular graph the two Cheeger constants are related by the following -
\begin{lemma}\label{Lem-vertex-edge}
		Let $\mathbb{G} = (V,E)$ be a $d$-regular graph
	$$ \frac{h(\mathbb{G})}{d} \leqslant \mathfrak{h}(\mathbb{G}) \leqslant h(\mathbb{G}).$$

\end{lemma}
\begin{proof}
	Let $V_{1}\subset V$ and we consider the map $$\psi:E(V_{1},V\backslash V_{1}) \rightarrow \delta(V_{1}) \text{ given by }(v_{1},s)\rightarrow v_{1}s.$$ 
	The map is surjective hence we have the left hand side and at the worst case $d$ to $1$ wherein we get the right hand side.
\end{proof}
We have the following inequalities, called the discrete Cheeger-Buser inequality. It is the version for graphs of the corresponding inequalities for the Laplace-Beltrami operator on closed Riemannian manifolds. It was first proven by Cheeger \cite{myfav77} (lower bound) and by Buser \cite{myfav78} (upper bound). The discrete version was shown by Alon and Millman  \cite{myfav79} (Proposition \ref{chin}).
\begin{proposition}[Discrete Cheeger-Buser inequality]\label{chin}
	Let $\mathbb{G} = (V,E)$ be a finite $d$-regular graph. Let $\lambda_{2}$ denote the second smallest eigenvalue of its normalised Laplacian matrix and $\mathfrak{h}(\mathbb{G})$ be the (edge) Cheeger constant. Then 
	$$ \frac{\mathfrak{h}(\mathbb{G})^{2}}{2} \leqslant \lambda_{2} \leqslant 2\mathfrak{h}(\mathbb{G}).$$
\end{proposition}
\begin{proof}
	See \cite{myfav65} prop. $4.2.4$ and prop. $4.2.5$ or \cite{myfav68} sec. $3$.
\end{proof}

Before proceeding further, let us recall the notion of Cayley graph of a group. 
\begin{definition}[Cayley graph]
	Let $G$ be a finite group and $S$ be a symmetric generating set of $G$. Then the Cayley graph $C(G,S)$ is the graph having the elements of $G$ as vertices and $\forall x,y\in G$ there is an edge between $x$ and $y$ if and only if $\exists s\in S$ such that $sx=y$. 
	If $1\in S$, then the graph has a loop (which we treat as an edge) going from $x$ to itself $\forall x \in G$. 
\end{definition}
 A graph is said to be $r$-regular (where $r\geqslant 1$ is an integer) if there are exactly $r$ half edges connected to each vertex (except for a loop which counts as one half edge). If $|S|= d$, it is clear that $C(G,S)$ will be $d$-regular (where $|S|$ denotes the cardinality of the set $S$).\\

Next, we recall the definition of the adjacency matrix associated to any finite undirected graph. For any finite undirected graph $\mathcal{G}$ having vertex set $V = \lbrace v_{1},...,v_{|\mathcal{G}|}\rbrace$ and edge set $E$, the adjacency matrix $T$ is the $|V|\times |V|$ matrix having $T_{ij}=$ the number of edges connecting $v_{i}$ with $v_{j}$. The discrete Cheeger inequality applies to all finite regular graphs (the inequality also holds for finite non-regular graphs where we need to consider the maximum of the degrees of the all the vertices - see \cite{myfav65} prop. $4.2.4$, but for our purposes we shall restrict to regular graphs).\\

We show the following proposition - 
\begin{proposition}[Criteria for non-bipartite property]\label{prop}
	A finite Cayley graph $C(G,S)$ is non-bipartite if and only if there does not exist an index two subgroup $H$ of $G$ which is disjoint from $S$. 
\end{proposition}
\begin{proof}
	 Let $C(G,S)$ be bipartite. Then we can partition the vertex set $G$ into two disjoint sets $A$ and $B$ such that $G = A\sqcup B.$ Let $\mathbf{1} \in B$. Let $s\in S\cap B $. Then $s^{-1}\in S$ and so $1=ss^{-1}\in A$. This is a contradiction. So $S\cap B = \phi$.\\
	 Now suppose $x,y\in B$ but $xy\notin B$. So $xy\in A$. Thus there exists $s_{1},s_{2},\cdots,s_{2r+1} \in S ,r\in \mathbb{N}$ such that $s_{1}s_{2}\cdots s_{2r+1}(xy) = y$. This implies that $s_{1}s_{2}\cdots s_{2r+1}x = 1 \in B$. But this is impossible because $x\in B$ so $s_{1}s_{2}\cdots s_{2r+1}x\in A$. Thus we have a contradiction and $xy \in B$. So, $B$ is an index $2$ subgroup disjoint from $S$. \\
	The other direction is clear.
\end{proof}

\begin{lemma}\label{lem1}
	Let $G$ be a finite group and $C(G,S)$ denote its Cayley graph with respect to a symmetric set $S$ of size $d$. Let $S$ be such that
	  $$ |SA\backslash A| \geqslant \epsilon'|A| \,\text{ ($\epsilon'$-combinatorial expansion of $S$)}$$
	for every set $A\subseteq G$ with $|A| \leqslant \frac{|G|}{2}$ and some $\epsilon' > 0$.
	 Then we have the estimate 
	$$|SA\backslash A|\geqslant \frac{\epsilon'}{d}|G\backslash A|$$
	for all sets $A\subseteq G$ with $|A|\geqslant \frac{|G|}{2}$.
\end{lemma}
\begin{proof} Let $A^{c} = G\backslash A$. The proof is based on the fact that $|SA\backslash A| \geqslant \frac{1}{d}|SA^{c}\backslash A^{c}|$ for all subsets $A\subseteq G$ and $S=S^{-1}\subset G$. \\ Let $s\in S$,
	$$|sA^{c}\cap A| = |s^{-1}(sA^{c}\cap A)| = |A^{c}\cap s^{-1}A| \leqslant |A^{c}\cap SA|$$
	$$\Rightarrow |SA^{c}\backslash A^{c}| = |SA^{c}\cap A| = |\cup_{s\in S}sA^{c}\cap A|\leqslant \Sigma_{s\in S}|SA\cap A^{c}|=d|SA\backslash A|. $$
	Hence, we have $$|SA\backslash A|\geqslant \frac{1}{d}|SA^{c}\backslash A^{c}| \geqslant \frac{\epsilon'}{d}|A^{c}|=\frac{\epsilon'}{d}|G\backslash A|.$$
	(Using the property of combinatorial expansion of $S$ and noting that $|A|\geqslant \frac{|G|}{2} \Rightarrow |A^{c}|\leqslant \frac{|G|}{2}$).
\end{proof}

To prove Theorem \ref{mainthm} we have to show that, under the given assumptions, we have 
$$t_{n}\geqslant -1+ \frac{h(G)^{4}}{\alpha d^{6}(d+1)^{2}},$$ 
for some absolute constant $\alpha$ (which we shall precise).\\

	\textbf{Method of Proof :} The proof is based on the following strategy. We shall first fix a small real number $\zeta$ (depending on the degree $d$ and expansion $\epsilon$) and suppose that the Cayley graph $C(G,S)$ has an eigenvalue less than $-1+\zeta$. Under this condition, we shall obtain a set $A$, such that $|A|$ is close to $\frac{|G|}{2}$ and which satisfies certain properties when we take translates of $A$ by elements $s\in S$. This is Lemma \ref{lem2}. Using this set $A$, we shall construct a subgroup $H$ of $G$ whose index will be $2$ (when $\zeta$ is small enough) and we shall show that this $H$ cannot intersect the generating set $S$ of $G$. This will give the required contradiction with Proposition \ref{prop}, since the Cayley graph $C(G,S)$ was non-bipartite.
\begin{lemma}\label{lem2}
	Let $G$ be a finite group, $k\geqslant 1$ and $S=S^{-1} = \lbrace s_{1},...,s_{d}\rbrace $ be a symmetric generating set of $G$. Let $S$ be $\epsilon$-combinatorially expanding, i.e.,
	$$|SX \backslash X| \geqslant \epsilon|X|$$ 	for every set $X\subseteq G$ with $|X| \leqslant \frac{|G|}{2}$ and some $\epsilon > 0$.\footnote{ It is clear that $d \geqslant \epsilon$ and in fact, considering $X\subset S,|X|\leqslant \frac{|G|}{2}$ we get that $d > \epsilon $, so that $\frac{\epsilon}{d}$ always remains strictly less than $1$ for finite Cayley graphs.} Suppose, there exists a sufficiently small $\zeta, 0<\zeta\leqslant \frac{\epsilon^{2}}{4d^{4}} $, such that the adjacency matrix $T$ of $C(G,S)$ has an eigenvalue in $[-1,-1+\zeta)$. Fix $\beta = d^{2}\sqrt{2\zeta(2-\zeta)} $. 
 Then, there exists a set $A$ with the following properties
	\begin{enumerate}
		\item $(\frac{1}{2 + \beta+ \frac{ d\beta}{\epsilon}})|G| \leqslant |A|\leqslant \frac{1}{2}|G|$,
		\item $|SA\cap A| \leqslant \frac{1}{\epsilon}\beta |A|$,
		\item  $\forall s\in S, g\in G, 	|sAg\Delta(Ag)^{c}|\leqslant \beta\Big(1+ \frac{ d}{\epsilon} +\frac{2}{\epsilon}\Big)|A|.$
	\end{enumerate}
\end{lemma}
\begin{proof}
		We have 
	
	\begin{equation}\label{eqn3.1}
	\epsilon|X|\leqslant |SX\backslash X|,
	\end{equation}
	whenever $X\subset G$ with $|X|\leqslant \frac{|G|}{2}$ and using Lemma \ref{lem1} with $|S| = d$
	
	\begin{equation}\label{eqn3.2}
	\frac{\epsilon}{d}|G\backslash X|\leqslant |SX\backslash X| ,
	\end{equation}
	whenever $|X| \geqslant \frac{|G|}{2}$.\\

	Since $T$ has an eigen-value in $[-1,-1+\zeta)$, $T^{2}$ has a non-trivial eigenvalue (say) $t'$ in $((1-\zeta)^{2},1]$.\footnote{actually we only need the fact that $t' > (1-\zeta)^{2}$. That $t' \neq 1$ follows when we consider non-bipartite graphs, since a graph is bipartite iff $T$ has $-1$ as an eigenvalue. }\\
	
		Now consider the set $S^{2}$ (obtained by identifying all equal elements in the multi-set $S.S$) and the muti-set $S' = S.S$ (without identification). $T^{2}$ is the adjacency matrix associated with $S'$ and $|S^{2}|\leqslant |S'|=d^{2}$. Let $h(G,S')$ denote the vertex Cheeger constant (Definition \ref{vertex-Cheeger}) and $\mathfrak{h}(G,S')$ denote the edge-Cheeger constant (Definition \ref{Edge-Cheeger}) for $G$ with respect to the multi-set $S'$.

	We have $t'> (1-\zeta)^{2}$. Let $\mathbb{L}$ denote the Laplacian matrix of the graph of $G$ with respect to $S'$, with the adjacency operator $T^{2}$ and let its eigenvalues be denoted by $0=\mathbf{L}_{1} \leqslant \mathbf{L}_{2}\leqslant ... \leqslant \mathbf{L}_{n} \leqslant 2$. We know that 
	$$\mathbf{L}_{2} = 1 - t' < 1 - (1-\zeta)^{2} = \zeta(2-\zeta).$$

	By the discrete Cheeger-Buser inequality (Proposition \ref{chin}) for the graph of $G$ with respect to $S'$ we have 
	
	$$\frac{\mathfrak{h}^{2}(G,S')}{2}\leqslant \mathbf{L}_{2}< \zeta(2-\zeta).$$ 
	Hence by Lemma \ref{Lem-vertex-edge},
	
	$$\frac{h(G,S')}{d^{2}} \leqslant   \mathfrak{h}(G,S') < \sqrt{2\zeta(2-\zeta)}.$$ 
	
	This implies that $\exists A \subset G$ with $ |A| \leqslant \frac{|G|}{2}$ such that
	\begin{equation}\label{eqn3.3}
	\frac{|S^{2}A\backslash A|}{|A|} \leqslant \frac{|S'A\backslash A|}{|A|} < d^{2}\sqrt{2\zeta(2-\zeta)} = \beta.
	\end{equation} 
	\begin{claim}
		$|A\cup SA| \geqslant \frac{|G|}{2}$ for $\zeta \leqslant \frac{\epsilon^{2}}{4d^{4}}$.
	\end{claim}
\begin{proof}[Proof of claim]
		We know that for arbitrary sets $X,Y,Z \subset G$, $X(Y\cup Z) = XY\cup XZ$. Hence 
	
	$$|S(A\cup SA)\backslash (A\cup SA)| = |S^{2}A\backslash A| <   d^{2}\sqrt{2\zeta(2-\zeta)}|A|.$$

	Let $|A\cup SA| \leqslant \frac{|G|}{2}$. This implies (using equation \ref{eqn3.1} and \ref{eqn3.3}) that $$\epsilon|A| \leqslant \epsilon|A\cup SA|\leqslant|S(A\cup SA)\backslash (A\cup SA)|<  d^{2}\sqrt{2\zeta(2-\zeta)}|A|,$$
	which cannot hold for $\zeta \leqslant \frac{\epsilon^{2}}{4d^{4}}$.

\end{proof}
	This means that under the assumption $\mathbf{\zeta \leqslant \frac{\epsilon^{2}}{4d^{4}}}$ we have $|A\cup SA| \geqslant \frac{|G|}{2}$.\\
	
	We can apply Lemma \ref{lem1} to $|A\cup SA|\geqslant \frac{|G|}{2}$  and use equation \ref{eqn3.2} and equation \ref{eqn3.3} to get
	
		$$\frac{\epsilon}{d}|G\backslash (A\cup SA)|\leqslant |S(A\cup SA)\backslash (A\cup SA)|= |S^{2}A\backslash A| <\beta|A|.$$
	Noting the fact that $|G\backslash (A\cup SA)| = |G|-|A\cup SA|$, we have

	$$ |G| - \frac{ d\beta}{\epsilon}|A| \leqslant |A\cup SA|\leqslant |A| + |SA|.$$

	We use the fact that, $$|SA|\leqslant |S^{2}A|\leqslant |A|+ \beta|A|$$
	to conclude that, 
	\begin{equation}\label{eqn3.4}
	\mathbf{ \left(\frac{1}{2 + \beta+ \frac{ d\beta}{\epsilon}}\right)|G|\leqslant |A|}.
	\end{equation}
	
	For arbitrary sets $X,Y,Z \subset G$ we have $X(Y\cap Z) \subset XY\cap XZ.$ \\
	
	Hence 		
	$$|S(A\cap SA)\backslash(A \cap SA)| \leqslant |S^{2}A\backslash A| \leqslant  \beta|A|.$$
	
	As $|A|\leqslant \frac{|G|}{2}$ clearly $|A\cap SA| \leqslant \frac{|G|}{2}$. So, the hypothesis of $\epsilon$-combinatorial expansion applies to $A\cap SA$ (i.e., $\epsilon|A\cap SA|\leqslant |S(A\cap SA)\backslash(A \cap SA)| \leqslant \beta|A|$) and we have
	\begin{equation}\label{eqn3.5}
	\mathbf{|A\cap SA| \leqslant \frac{1}{\epsilon}\beta |A|.}
	\end{equation}
	Our next aim is to compute the bounds on $|sA\Delta A|, |sA\Delta A^{c}|, |sAg\Delta Ag|, |sAg\Delta(Ag)^{c}|$ for $g\in G$.

	For this,
	\begin{align*}
	|sA\Delta A| &= |sA\cup A\backslash sA\cap A| \\
	& = |sA\cup A|- |sA\cap A| \\
	& = |sA|+|A| - 2|sA\cap A| \\
	& = 2|A| - 2|sA\cap A| \\
	& \geqslant 2|A| - 2|SA\cap A|\\
	& \geqslant \left(2-\frac{2\beta}{\epsilon}\right)|A|.
	\end{align*}

	This implies, 
	\begin{align*}
			|sA\Delta A^{c}| & = |G\setminus (sA\Delta A) | \\
			& = |G| - |sA\Delta A|\\
			& \leqslant |G| - (2-\frac{2}{\epsilon}\beta)|A|\\
			& \leqslant \Big(\beta+ \frac{ d\beta}{\epsilon} +\frac{2}{\epsilon}\beta\Big)|A|.
	\end{align*}

	Thus we have \\
	\begin{equation}
	|sA\Delta A^{c}| \leqslant \beta\left(1+ \frac{ d}{\epsilon} +\frac{2}{\epsilon} \right)|A|
	\end{equation}
	and by the symmetricity of $S$,
	$$|sA^{c}\Delta A| \leqslant \beta\left(1+ \frac{ d}{\epsilon} +\frac{2}{\epsilon} \right)|A|.$$

	Now let $g\in G$ be arbitrary. Then we have, 
	\begin{align*}
			|sAg\Delta Ag| & = |sAg|+|Ag| - 2|sAg\cap Ag| \\
			& = 2|A| -2|sA\cap A| \\
			& \geqslant \left(2-\frac{2}{\epsilon}\beta\right)|A|.
	\end{align*}
	(Since for fixed $g\in G$, $X,Y\subset G$, $(X\cap Y)g = Xg\cap Yg$).\\
	
	Similarly, we get	
	\begin{equation}
	\mathbf{|sAg\Delta(Ag)^{c}|\leqslant \beta\left(1+ \frac{ d}{\epsilon} +\frac{2}{\epsilon}\right)|A|}
	\end{equation}
	and 
	$$|s(Ag)^{c}\Delta Ag|\leqslant \beta\left(1+ \frac{ d}{\epsilon} +\frac{2}{\epsilon}\right)|A|.$$

\end{proof}

	We shall now use the set $A$ which we obtained from the lemma to prove our main theorem.
	
%\pagebreak

\begin{theorem}\label{mainthm2}
	Let $G$ be a finite group, $k\geqslant 1$ and $S=S^{-1} = \lbrace s_{1},...,s_{d}\rbrace $ be a symmetric generating set of $G$. Suppose that $G$ does not have an index two subgroup $H$ disjoint from $S$. Let $S$ be $\epsilon$-combinatorially expanding, i.e.,
	$$|SX \backslash X| \geqslant \epsilon|X|$$ 	for every set $X\subseteq G$ with $|X| \leqslant \frac{|G|}{2}$ and some $\epsilon > 0$. 
	Then all the eigenvalues of the operator $T$ are $ \geqslant -1 +\frac{\epsilon^{4}}{\alpha d^{6}(d+1)^{2}}$ where $\alpha$ is an absolute constant (we can take $\alpha = 2^{9} $).
\end{theorem}
\begin{proof}
	The proof will be by contradiction. Keeping the notations of Lemma \ref{lem2}, we shall show that if $T$ has an eigenvalue in $[-1,-1+\zeta)$, where $\zeta$  is chosen to be small (precised in Claim \ref{clmthm} and satisfying the condition on $\zeta$ in Lemma \ref{lem2}), there exists an index $2$ subgroup, $H$ in $G$ which is disjoint from $S$. This will give the required contradiction.\\
	
	First we use Lemma \ref{lem2} to conclude that (under the assumption $\zeta\leqslant \frac{\epsilon^{2}}{4d^{4}}$) there exists a set $A$ with the following properties
	\begin{enumerate}
		\item $(\frac{1}{2 + \beta+ \frac{ d\beta}{\epsilon}})|G| \leqslant |A|\leqslant \frac{1}{2}|G|,$
		\item $|SA\cap A| \leqslant \frac{1}{\epsilon}\beta |A|$,
		\item  $\forall s\in S, |sA\Delta A^{c}| \leqslant \Big(\beta+ \frac{ d\beta}{\epsilon} +\frac{2}{\epsilon}\beta\Big)|A|, |sA^{c}\Delta A| \leqslant \Big(\beta+ \frac{ d\beta}{\epsilon} +\frac{2}{\epsilon}\beta\Big)|A|$,
		\item  $\forall s\in S, g\in G, 	|sAg\Delta(Ag)^{c}|\leqslant \beta\Big(1+ \frac{ d}{\epsilon} +\frac{2}{\epsilon}\Big)|A|, |s(Ag)^{c}\Delta Ag|\leqslant \beta\Big(1+ \frac{ d}{\epsilon} +\frac{2}{\epsilon}\Big)|A|.$\\
	\end{enumerate}

		Using the above set $A$, we want to construct the subgroup $H$ of index $2$. The method will be to translate $A$ using the elements $g\in G$ and check which of them have large intersection with the original set $A$ (i.e, $|A\cap Ag|$ is ``almost" $|A|$).\\

		Take $A_{g} := A\cap Ag, A_{g}' := (A\cup Ag)^{c}$. 	Let $B = A_{g}\sqcup A_{g}'$ (it is a disjoint union). Then 
	
			$$G\setminus B  = B^{c} = A\Delta Ag$$
			and 
		$$ 	B  = (A\Delta Ag)^{c} = A\Delta (Ag)^{c}.$$

		 Also note that $X\Delta Y  = X^{c}\Delta Y^{c}$ for all $X,Y \subseteq G$.\\

	We wish to estimate $|B|$ when $g\in G$. For this, we first estimate $|SB\Delta B|$ and $|SB^{c}\Delta B^{c}|$.
	\begin{align*}
	|SB\Delta B| &\leqslant \Sigma_{s\in S}|sB\Delta B|  \\
	& = \Sigma_{s\in S}|s(A\Delta (Ag)^{c})\Delta(A\Delta (Ag)^{c})|\\
	& = \Sigma_{s\in S}|(sA\Delta s(Ag)^{c})\Delta (A\Delta (Ag)^{c})|\\
	& = \Sigma_{s\in S}|(sA\Delta s(Ag)^{c})\Delta(A^{c}\Delta Ag)|\\
	& =  \Sigma_{s\in S}|(sA\Delta A^{c})\Delta(s(Ag)^{c}\Delta (Ag))|\\
	& \leqslant d(|sA\Delta A^{c}| + |sAg\Delta(Ag)^{c}|)\\
	& \leqslant 2d\beta\Big(1+ \frac{ d}{\epsilon} +\frac{2}{\epsilon}\Big)|A|.
	\end{align*}
	(where we use the fact that, all the above sets are defined inside $G$, $X\Delta Y = X^{c}\Delta Y^{c}$, $sX^{c} = (sX)^{c}$, $s(X\Delta Y) = (sX\Delta sY)$ \footnote{These do not hold for sets $S\subset G$ in general, i.e., $S.X^{c}\neq (SX)^{c}$ and $SX\Delta SY \subset S(X\Delta Y)$ for arbitrary sets $S,X,Y\subset G$. This is one of the main reasons why we had to estimate translates of $A$ by elements $s \in S$ rather than translate of $A$ by $S$. } and $\Delta$ is both associative and commutative).
	
	Similarly,
	\begin{align*}
			|SB^{c}\Delta B^{c}| & \leqslant \Sigma_{s\in S}|sB^{c}\Delta B^{c}| \\
			& = \Sigma_{s\in S}|s(A\Delta Ag)\Delta (A\Delta Ag)| \\
			& = \Sigma_{s\in S}|(sA\Delta sAg)\Delta (A\Delta Ag)|\\
			& = \Sigma_{s\in S}|(sA\Delta sAg)\Delta (A^{c}\Delta (Ag)^{c})|\\
			& = \Sigma_{s\in S}|(sA\Delta A^{c})\Delta (sAg\Delta (Ag)^{c})|\\
			 & \leqslant  d(|sA\Delta A^{c}| + |sAg\Delta(Ag)^{c}|)\\
			 & \leqslant 2d\beta\Big(1+ \frac{ d}{\epsilon} +\frac{2}{\epsilon}\Big)|A|.
	\end{align*}

	We now have the following two cases depending on the size of the set $B$.
	\begin{enumerate}
		\item $|B|\leqslant \frac{|G|}{2}$ in which case,
		
		\begin{equation}\label{eqn3.11}
			|B| \leqslant \frac{2d\beta}{\epsilon^{2}}\Big(\epsilon+  d +2\Big)|A|  
		\end{equation}
		(using the fact that $\epsilon|B| \leqslant |SB\backslash B| \leqslant |SB\Delta B|\leqslant   2d\beta\Big(1+ \frac{ d}{\epsilon} +\frac{2}{\epsilon}\Big)|A| $).\\
			From this it follows that, 
		\begin{equation}
		|A\cap Ag| \leqslant  \frac{d\beta}{\epsilon^{2}}\Big(\epsilon+  d +2\Big)|A|. 
		\end{equation}
		(There are two ways to see it. $B = A_{g}\sqcup A_{g}'$ and $|A_{g}'| \geqslant |A_{g}|$ when $|A|\leqslant \frac{|G|}{2} \Rightarrow |A\cap Ag| = |A_{g}|\leqslant \frac{|B|}{2}$, is one way. The other way is to use $G\setminus B = A\Delta Ag$. Hence after taking the cardinalities and expanding we have $|B| = |G|-2|A| + 2|A\cap Ag|$. Then use the fact that $2|A|\leqslant |G|$, to get that $|A\cap Ag|\leqslant \frac{|B|}{2}$.)\\
		
		 \textbf{OR}\\
		\item $|B|> \frac{|G|}{2}$ in which case $|B^{c}|\leqslant\frac{|G|}{2}$ and then 
		\begin{equation} \label{eqn3.12}
		|G\backslash B|\leqslant \frac{2d\beta}{\epsilon^{2}}\Big(\epsilon+  d +2\Big)|A|
		\end{equation}
		(using the fact that $\epsilon|B^{c}| \leqslant |SB^{c}\backslash B^{c}| \leqslant |SB^{c}\Delta B^{c}| \leqslant  2d\beta\Big(1+ \frac{ d}{\epsilon} +\frac{2}{\epsilon}\Big)|A| $).
		\\
		From this it follows that, 
		\begin{equation}
		\Big( 1 -\frac{d\beta}{\epsilon^{2}}\Big(\epsilon+  d +2\Big)\Big)|A|\leqslant |A\cap Ag|.
		\end{equation}
		(Again, using  $G\setminus B = A\Delta Ag$, taking the cardinalities and expanding the expression we have $|G\setminus B| = 2|A| - 2|A\cap Ag|$).\\
	\end{enumerate}

    Thus for any $g\in G$, we have either
    \begin{enumerate}\label{enum}    
    	\item [(i)] $|A\cap Ag| \leqslant \frac{d\beta}{\epsilon^{2}}\Big(\epsilon+  d +2\Big)|A|,$ \\
    		
    	\textbf{OR}\\
    	
    	\item [(ii)] $|A\cap Ag| \geqslant \Big( 1 -\frac{d\beta}{\epsilon^{2}}\Big(\epsilon+  d +2\Big)\Big)|A|$.\\
    \end{enumerate}

   The trick now is to use the method of Freiman in \cite{myfav67} to find a subgroup $H$ of $G$. We prove it in the following claim.
   \begin{claim}\label{clmthm}
   	If $H:= \lbrace g\in G : |A\cap Ag|\geqslant r|A| \rbrace $ 
   	where $r =1 -\frac{d\beta}{\epsilon^{2}}\Big(\epsilon+  d +2\Big)$ and $\beta \leqslant \frac{1}{2^{3}\sqrt{2}}\times \frac{\epsilon^{2}}{d(d+1)} $, then $H$ is a subgroup of $G$ of index $2$.
   \end{claim}
\begin{proof}[Proof of claim]
	 We have $H=H^{-1}$, $1\in H$ and $r> \frac{1}{2} + \frac{d\beta}{\epsilon^{2}}\Big(\epsilon+  d +2\Big)$. Also for $g,h \in H$ we have by the triangle inequality
	 \begin{align*}
	 		|A\backslash Agh|& \leqslant |A\backslash Ah| + |Ah\backslash Agh|\\
	 		& \leqslant 2(1-r)|A|.  		
	 \end{align*}
	 This implies, 
	 	 		 $$ |A\cap Agh| \geqslant (2r - 1)|A|.$$

	Hence, $gh$ cannot belong to case (i), $gh$ belongs to case (ii), i.e., $gh\in H$. So $H$ is a subgroup of $G$. 	
	
	Let $z= \frac{d\beta}{\epsilon^{2}}\Big(\epsilon+  d +2\Big)$. Using the estimate,
	\begin{align*}
			|A|^{2} & = \Sigma_{g\in G}|A\cap Ag| \\
			& \leqslant |H||A| + \frac{d\beta}{\epsilon^{2}}\Big(\epsilon+  d +2\Big)|A||G\backslash H|, 
	\end{align*}	
	we have 
	\begin{align*}
			& |A| \leqslant |H| + z(|G|- |H|), \\
	\end{align*}
	which implies that,
				 $$\left(\frac{1}{2 + \beta+ \frac{ d\beta}{\epsilon}}\right)|G| - z|G|  \leqslant (1-z)|H|.$$
	(Using the fact that $ \left(\frac{1}{2 + \beta+ \frac{ d\beta}{\epsilon}}\right)|G|\leqslant |A|$). \\
	
	The index of $H$ in $G$ is $2$ if $|H|> \frac{|G|}{3}$ and thus, to conclude that $H$ is a subgroup of $G$ of index $2$, it suffices to show that \footnote{Note that $H \neq G$ since there are elements $g\in G$ such that, $g\in G\setminus H$. }
	\begin{align*}
		\left(\frac{1}{2 + \beta+ \frac{ d\beta}{\epsilon}} - z \right) & >\frac{1-z}{3} \\
		\Leftrightarrow \,\,\, \frac{1}{\left(2 + \beta+ \frac{ d\beta}{\epsilon} \right)} & > \frac{1+2z}{3}.
	\end{align*}

	Substituting the expression for $z$, it suffices to show that,
	$$\left(2 + \beta+ \frac{ d\beta}{\epsilon}\right)  + \frac{2d\beta}{\epsilon^{2}}\left(\epsilon+  d +2\right)\left(2 + \beta+ \frac{ d\beta}{\epsilon}\right) < 3,$$
	i.e., 
	$$\left(\beta+ \frac{ d\beta}{\epsilon}\right)  + \frac{2d\beta}{\epsilon^{2}}\left(\epsilon+  d +2\right)\left(2 + \beta+ \frac{ d\beta}{\epsilon} \right) < 1.$$
	Now, using the fact that $\beta < \frac{1}{8\sqrt{2}}, \frac{d\beta}{\epsilon} < \frac{1}{8\sqrt{2}}, \frac{2d\beta}{\epsilon^{2}} < \frac{1}{4\sqrt{2}(d+1)}, \epsilon < d,$ $\frac{1}{4\sqrt{2}} < 0.177$, we have,
	\begin{align*}
			& \left(\beta+ \frac{ d\beta}{\epsilon}\right)  + \frac{2d\beta}{\epsilon^{2}}\Big(\epsilon+  d +2\Big)\left(2 + \beta+ \frac{ d\beta}{\epsilon}\right) \\
			& < \left( \frac{1}{8\sqrt{2}} + \frac{1}{8\sqrt{2}}\right) + \frac{1}{4\sqrt{2}(d+1)}(2d+2)\left(2 +  \frac{1}{8\sqrt{2}} + \frac{1}{8\sqrt{2}}\right)\\
			& = \frac{1}{4\sqrt{2}} + \frac{d+1}{d+1}.\frac{2+\frac{1}{4\sqrt{2}}}{2\sqrt{2}} \\
			& < 0.177 + 0.77 \\
			& < 1.	
	\end{align*}	
	Hence the index of $H$ in $G$ is $2$ if $d^{2}\sqrt{2\zeta(2-\zeta)} = \beta \leqslant \frac{\epsilon^{2}}{2^{3}\sqrt{2}.d(d+1)}$. This gives us the fact that, for all $\zeta \leqslant \frac{\epsilon^{4}}{2^{9}d^{6}(d+1)^{2}}$, the index of $H$ in $G$ is $2$.
\end{proof}

	Up until now, we have argued formally that under the condition on $\zeta$ (equivalently $\beta$) being small enough for the set $A$ to exist (essentially $\beta < \epsilon$ or $\zeta \leqslant \frac{\epsilon^{2}}{4d^{4}}$). From the above claim, we see that all $\zeta \leqslant \frac{\epsilon^{4}}{2^{9}d^{6}(d+1)^{2}}$ satisfies this condition (since $\frac{\epsilon}{d} < 1 $). From now on fix $\zeta (> 0)$ to be any real number $\leqslant \frac{\epsilon^{4}}{2^{9}d^{6}(d+1)^{2}}$.\\

	We have found an index two subgroup $H$ in $G$. We shall now show that this subgroup $H$ is disjoint from $S$.\\
	
	Suppose, on the contrary that $t\in S\cap H$. This means the following
	\begin{itemize}
		\item $t\in S$. Therefore, $|tA\cap A|\leqslant |SA\cap A| \leqslant \frac{\beta}{\epsilon}|A|$ (see Lemma \ref{lem2}).
		\item $t\in H$. Therefore, by definition of $H$, $|tA\cap A|\geqslant r|A|$, where $r =1 -\frac{d\beta}{\epsilon^{2}}\Big(\epsilon+  d +2\Big)$. 
		
	\end{itemize} 
	   Combining, we see that $r\leqslant \frac{\beta}{\epsilon}$. This is clearly a contradiction since $0.82 < (1-\frac{1}{4\sqrt{2}})\leqslant r $ and $\frac{\beta}{\epsilon} < \frac{1}{8\sqrt{2}} < 0.09$.
  
   This implies that $S\subset G\backslash H$, contradicting the hypothesis.\\
   
    To summarise, we have shown that - for any fixed $\zeta \leqslant \frac{\epsilon^{4}}{2^{9}d^{6}(d+1)^{2}}$, if there exists an eigenvalue of the normalised adjacency matrix of $C(G,S)$ less than $-1+\zeta$, then $C(G,S)$ must be bipartite (equivalently it has an index $2$ subgroup disjoint from $S$). That means, for non-bipartite $C(G,S)$, we must have all eigenvalues of the normalised adjacency matrix $\geqslant -1 +\frac{\epsilon^{4}}{\alpha d^{6}(d+1)^{2}}$ with 
   $\alpha = 2^{9}$. We are done.

\end{proof}

Since, by definition, the vertex Cheeger constant $h(G)$ is the infimum of $\frac{|SX\setminus X|}{|X|}$, we can replace $\epsilon$ by $h(G)$ in the above arguments, thus proving Theorem \ref{mainthm}.

\section{Concluding Remarks}

The above bound is dependent on the Cayley graph structure and does not hold for general non-bipartite finite, regular graphs. In the setting of arbitrary finite regular graphs some recent works are worth mentioning. Bauer and Jost in \cite{myfav69} introduced a dual Cheeger constant $\bar{h}$ which encodes the bipartiteness property of finite regular graphs. The dual Cheeger constant $\bar{h}$ of a $d$ regular graph is defined as 
$$ \bar{h} := \max_{V_{1},V_{2},V_{1}\cup V_{2}\neq \phi} \frac{2|E(V_{1},V_{2})|}{vol(V_{1})+vol(V_{2})},$$ for a partition $V_{1},V_{2},V_{3}$ of the vertex set $V$, $vol(V_{k}) = d|V_{k}|$ and $|E(V_{1},V_{2})|$ denotes the number of edges going from $V_{1}$ into $V_{2}$. For a general regular graph it was shown by Bauer-Jost (and independently by Trevisan \cite{myfav70}) that 

\begin{theorem}[Bauer-Jost \cite{myfav69}]
	Let $\lambda_{n}$ be the largest eigen-value of the graph laplace operator. Then $\lambda_{n}$ satisfies  
	$$\frac{(1-\bar{h})^{2}}{2}\leqslant 2-\lambda_{n} \leqslant 2(1-\bar{h})$$
	and the graph is bipartite if and only if $\bar{h} = 1$.
\end{theorem}

There is also the concept of higher order Cheeger constants introduced by Miclo in \cite{myfav91}.

Some recent works treating higher order Cheeger inequalities for general finite graphs are those by Lee--Gharan--Trevisan in \cite{myfav90} and Liu \cite{myfav72} (for the dual case) etc. 

\section*{Acknowledgements} I wish to thank Emmanuel Breuillard for the problem and also for a number of helpful discussions and advice on the subject. I wish to thank the anonymous referees for their comments and suggestions, thus improving the article. I also wish to acknowledge the support of the OWLF programme of the Mathematisches Forschunginstitut Oberwolfach.

\providecommand{\bysame}{\leavevmode\hbox to3em{\hrulefill}\thinspace}
\providecommand{\MR}{\relax\ifhmode\unskip\space\fi MR }
\providecommand{\MRhref}[2]{%
\href{http://www.ams.org/mathscinet-getitem?mr=#1}{#2}
}
\providecommand{\href}[2]{#2}

\end{document}